\documentclass{article}
\usepackage{latexsym,amssymb,
amsthm,amsfonts,amsmath}
\newtheorem{theorem}{Theorem}[section]
\newtheorem{lemma}[theorem]{Lemma}
\theoremstyle{definition}

\newtheorem{example}[theorem]{Example}
\newtheorem{remark}[theorem]{Remark}
\newtheorem{corollary}[theorem]{Corollary}

\def \N{{\mathbb N}}
\def \H{{\mathbb H}}
\def \E{{\mathbb E}}
\def \R{{\mathbb R}}
\newcommand{\vn}{\mathop{\mathrm{int}}\nolimits}
\newcommand{\rel}{\mathop{\mathrm{rel\:int}}\nolimits}

\title{\huge Reflection subgroups of Coxeter groups}

\author{Anna Felikson\thanks{Partially supported by grants NSh-5666.2006.1, INTAS YSF-06-10000014-5916, and  RFBR 07-01-00390-a.} 
\qquad\qquad Pavel Tumarkin\thanks{Partially supported by grants MK-6290.2006.1, NSh-5666.2006.1, INTAS YSF-06-10000014-5766, and RFBR  07-01-00390-a}\\[-0.5ex]
\small \texttt{felikson@mccme.ru}$\:$\qquad\qquad\qquad \texttt{pasha@mccme.ru}$\;\;$\qquad\\[0.9ex] 
\small Independent University of Moscow\\[-0.8ex] 
\small B. Vlassievskii 11, 119002 Moscow, Russia\\[0.8ex] 
\small Department of Mathematics, University of Fribourg\\[-0.8ex] 
\small P\'erolles, Chemin du Mus\'ee 23, CH-1700 Fribourg, Switzerland
}
\date{\small Mathematics Subject Classification: 20F55, 51M20, 51F15.}

\begin{document}

\maketitle

\begin{abstract}
We use geometry of Davis complex of a Coxeter group to investigate finite index reflection subgroups of Coxeter groups.
The main result is the following: if $G$ is an infinite indecomposable Coxeter group and $H\subset G$ is a finite index reflection
subgroup  then the rank of $H$ is not less than the rank of $G$. 
This generalizes results of~\cite{fund}. We also describe some properties of the nerves of the group and the 
subgroup in the case of equal ranks. 

\end{abstract}

\maketitle

\section{Introduction}

In~\cite{Da} M.~Davis constructed for any Coxeter system $\left(G,S\right)$
a contractible piecewise Euclidean complex, on which $G$ acts properly and cocompactly 
by reflections. In this paper, we use this complex to study finite index reflection 
subgroups of infinite indecomposable Coxeter groups from geometrical point of view.
We define convex polytopes in the complex to prove the following result:

\begin{theorem}
Let $\left(G,S\right)$ be a Coxeter system, where $G$ is infinite and indecomposable, and $S$ is finite.
If $P$ is a compact polytope in $\Sigma$,
then the number of facets of $P$ is not less than $|S|$. 

\end{theorem}

This generalizes results of~\cite{fund}, where the similar result was proved for fundamental polytopes of
finite index subgroups of cocompact (or finite covolume) groups generated by reflections in hyperbolic and Euclidean spaces.
M.~Dyer~\cite{D} proved that any reflection subgroup of a Coxeter group is also a Coxeter group. 
Using Theorem~\ref{facets} and criterium for finiteness of a Coxeter group provided by V.~Deodhar~\cite{De}, 
we obtain the main result of this paper: 

\begin{theorem}
Let $\left(G,S\right)$ be a Coxeter system, where
$G$ is infinite and indecomposable, and $S$ is finite.
Let $H\subset G$ be a finite index reflection subgroup.
Then any set of reflections generating $H$ contains at least $|S|$ elements.

\end{theorem}   
 
Further, we consider geometry of Davis complex itself. 
E.~M.~Andreev~\cite{Andr} obtained the following result for polytopes in hyperbolic and Euclidean spaces:

\begin{theorem}[Andreev~\cite{Andr}]
Let $P$ be  an acute-angled polytope in $\E^n$ or $\H^n$,
and let $a$ and $b$ be two faces of $P$. 
If $a\cap b=\emptyset$ then
the minimal planes containing, respectively, $a$ and $b$ are disjoint either. 

\end{theorem}

In Section~\ref{andr}, we define acute-angled polytopes in Davis complex and prove a counterpart of 
Andreev's theorem for non-intersecting facets of acute-angled polytopes (see Lemma~\ref{andr for facets}). 

\medskip

In Section~\ref{nn}, we focus on the case when a Coxeter group and its finite index reflection 
subgroup have the same rank, i.e. the rank of subgroup is minimum possible. 
We show that the nerve of the subgroup can be obtained by deleting some simplices from the nerve of
the group (Lemma~\ref{nerve}), and provide some necessary conditions for combinatorics of Coxeter group having a finite 
index reflection subgroup of the same rank (Lemma~\ref{comm}).

\medskip

We would like to thank E.~B.~Vinberg for useful comments and discussions, and
R.~B.~Howlett for communicating a proof of Lemma~\ref{generators}.
The work was mainly done during our stay at the University of Fribourg, Switzerland. 
We are grateful to the University for hospitality.

\section{Preliminaries}

\subsection{Coxeter systems}

A {\it Coxeter group} is a group with presentation
$$
\langle S \ | \  \left(s_is_j\right)^{m_{ij}}=1  \rangle
$$
for all $s_i,s_j\in S$, 
where $m_{ii}=1$ and $m_{ij}\in \{2,3,\dots,\infty\}$ for $i\ne j$. 
We further require $S$ to be finite.
$S$ is called a {\it standard generating set}.
A pair $\left(G,S\right)$ is called a {\it Coxeter system}.

If $S$ is fixed, $G$ is called
{\it decomposable} if $S=S_1\cup S_2$, 
where $S_1$ and $S_2$ are non-empty subsets such that $s_is_j=s_js_i$
for all $s_i\in S_1$, $s_j\in S_2$.
If  $G$ is not decomposable, it is called {\it indecomposable}.

An element $s\in G$ is called a {\it reflection} if it is conjugate to some of $s_i\in S$
(in particular, any  $s_i\in S$ is a reflection). A subgroup $H\subset G$ is called {\it reflection subgroup} of $G$ if $H$ is 
generated by reflections.

For any $T\subset S$ a reflection subgroup $G_T$ generated by all $s_i\in T$ is called a
{\it standard} subgroup of $G$. 

The {\it rank} of a Coxeter system $\left(G,S\right)$ is the number of reflections in $S$.
We denote it by $|S|$.

The proof of the following lemma is suggested by R.~B.~Howlett.

\begin{lemma}
\label{generators}
Let $\left(G,S\right)$ be a Coxeter system of rank $n$.
Then $G$ cannot be generated by 
less than $n$ reflections. 

\end{lemma}

\begin{proof}
Let $S=\{s_1,\dots,s_n\}$. 
Consider the Tits representation of $\left(G,S\right)$
 on a real vector $n$-space $V$ (see e.g.~\cite{B}). 
Suppose that $G$ is generated by $k<n$ reflections 
$r_1,r_2,\dots,r_k$                       
along vectors $v_1,v_2, \dots , v_k$, i.e.  
$$r_i(x)=x-2\frac{(v_i,x)}{(v_i,v_i)}v_i.$$

Let $L$ be the linear subspace spanned by $v_1,\dots,v_k$.
Let $g=r_{i_1}r_{i_2}\dots r_{i_l}$ be any element of $G$.
We prove by induction on $l$  that for any $v\in V$ \quad
$g(v) \in v + L$.
For $l=0$ the statement is evident (since $v\in v+L$).
Suppose the statement holds for all elements $g=r_{i_1}r_{i_2}\dots r_{i_{l-1}}$.
Let $g=r_{i_1}g_1$, where $g_1=r_{i_2}\dots r_{i_l}$.
Then $$g(v)=r_{i_1}g_1(v)=g_1(v)-2\frac{(v_{i_1},g_1(v))}{(v_{i_1},v_{i_1})}v_{i_1}\in 
v+L+ \lambda v_{i_1} \subset v+L$$
(where $\lambda\in \R$).
Hence, $g(v)\in v + L$ for any $g\in G$.

Suppose that $g\in G$ is a reflection along some vector $v$.
Then $-v = g(v) \in v + L$.
Therefore, $v\in L$. Hence,
by the construction of Tits representation of $\left(G,S\right)$,
$L$ should coincide with $V$, which
is impossible for a space spanned by $k<n$ vectors.

\end{proof}

\subsection{Davis complex}

For any Coxeter system $\left(G,S\right)$ there exists a contractible
piecewise Euclidean cell complex $\Sigma\left(G,S\right)$ (called {\it Davis complex}) 
on which $G$ acts discretely, properly and cocompactly. 
The construction was introduced by Davis~\cite{Da}.
In~\cite{M} Moussong proved that this complex yields
a natural complete piecewise Euclidean metric which is $CAT\left(0\right)$. 
We give a brief description of this complex following~\cite{VN}.

For a finite group $G$ the complex $\Sigma\left(G,S\right)$ is just one cell, which is obtained
as a convex hull $C$ of $G$-orbit of a suitable point $p$ in the standard linear 
representation of $G$ as a group generated by reflections. The point $p$ is chosen in such 
a way that its stabilizer in $G$ is trivial and all the edges of $C$ are of length 1.
The faces of $C$ are naturally identified with Davis complexes of the subgroups of $G$
conjugate to standard subgroups.

If $G$ is infinite, the complex $\Sigma\left(G,S\right)$ is built up of the Davis complexes of maximal 
finite subgroups of $G$ gluing together along their faces corresponding to common 
finite subgroups. The 1-skeleton of $\Sigma\left(G,S\right)$ considered as a combinatorial graph 
is isomorphic to the Cayley graph of $G$ with respect to the generating set $S$.

The action of $G$ on $\Sigma\left(G,S\right)$ is generated by reflections.
The {\it walls} in $\Sigma\left(G,S\right)$ are the fix point sets of reflections in $G$.
The intersection of a wall $\alpha$ with cells of $\Sigma\left(G,S\right)$ supply $\alpha$
with a structure of a piecewise Euclidean cell complex with finitely many isometry types of cells.
Walls are totally geodesic: 
any geodesic joining two points of $\alpha$ lies entirely  in $\alpha$. 
Since $\Sigma$ is $CAT\left(0\right)$, any two points of  $\Sigma$ may be joined by a unique geodesic. 

Any wall divides $\Sigma\left(G,S\right)$ into two connected components. All the walls decompose 
$\Sigma\left(G,S\right)$ into connected components which are compact sets called {\it chambers}. Any chamber is a 
fundamental domain of $G$-action on $\Sigma\left(G,S\right)$.  
The set of all chambers with appropriate adjacency relation is isomorphic 
to the Cayley graph of $G$ with respect to $S$.

A {\it nerve} of a Coxeter system $\left(G,S\right)$ is a simplicial complex 
with vertex set $S$. A collection of vertices span a simplex if and only if the corresponding reflections 
generate a finite group. It is easy to see that the combinatorics of chamber of $\Sigma\left(G,S\right)$ is 
completely determined by the nerve.   

For any subgroup $H\subset G$ we say that a wall $\alpha$ is a {\it mirror}
of $H$ if $H$ contains the reflection with respect to $\alpha$. 

In what follows, if $G$ and $S$ are fixed, we write $\Sigma$ instead of $\Sigma\left(G,S\right)$.

\subsection{Convex polytopes in $\Sigma$}

For any wall $\alpha$ of $\Sigma$ we denote by $\alpha^+$ and $\alpha^-$
the closures of the connected components of $\Sigma\setminus \alpha$, we call these components {\it halfspaces}.

A {\it convex polytope} $P\subset \Sigma$ is an intersection of finitely many halfspaces
$P=\bigcap\limits_{i=1}^n  \alpha_i^+$, such that $P$ is not contained in any wall.
Clearly, any convex polytope $P\subset \Sigma$ is a union of closed chambers.
$P$ is compact if and only if it contains finitely many chambers. 
Since the walls are totally geodesic, any convex polytope is convex
in usual sense: for any two points  $p_1,p_2\in P$
the geodesic segment connecting $p_1$ with $p_2$ belongs to $P$.

In what follows by writing $P=\bigcap\limits_{i=1}^n  \alpha_i^+$ we assume that
the collection of walls $\alpha_i$ is minimal: for any $j=1,\dots ,n$ we have
$P\ne \bigcap\limits_{i\ne j} \alpha_i^+$.
A {\it facet} of $P$ is an intersection $P\cap \alpha_i$ for some $i\le n$.
For any $I\subset \{1,\dots,n\}$ a set $\bigcap\limits_{i\in I}  \alpha_i$ is called a {\it face} of 
$P$ if it is not empty.

We can easily define a dihedral angle formed by two walls:
if $\alpha_i\cap \alpha_j\ne \emptyset$ there exists a maximal cell $C$  
of $\Sigma$ intersecting $\alpha_i\cap \alpha_j$.
We define the angle $\angle \left(\alpha_i,\alpha_j\right)$ to be equal to the corresponding Euclidean angle
formed by $\alpha_i\cap C$ and $\alpha_j\cap C$.
Clearly, any dihedral angle formed by two intersecting walls in $\Sigma$ is equal to
 $k\pi / m$ for some positive integers $k$ and $m$.  
A convex polytope $P$ is called {\it acute-angled} if each of the dihedral angles of $P$ 
does not exceed $\pi /2$. 

A convex polytope $P$ is called a {\it Coxeter polytope} if all its dihedral angles are
integer submultiples of $\pi$. Clearly, a fundamental domain of any reflection subgroup of $G$
is a Coxeter polytope in $\Sigma$.  Conversely, Theorem~4.4 of~\cite{D} implies that any Coxeter polytope in $\Sigma$ is a fundamental chamber for subgroup of $G$ generated by reflections in its walls.

\section{Proof of the theorems}

In what follows 
we write  $|P|$ for the number of facets of a convex polytope
$P\subset \Sigma$. 

Let $a_i$ and $a_j$ be intersecting facets of a convex polytope $P\subset \Sigma$.
We say that a dihedral angle of $P$ formed by $a_i$ and $a_j$  is {\it decomposed} 
if there exists a wall $\gamma$ containing $a_i\cap a_j$ and intersecting $\vn \left(P\right)$.

The following lemma is proved by V.~Deodhar.

\begin{lemma}[Deodhar~\cite{De}, Proposition~4.2]
\label{subgr}
Let $\left(G,S\right)$ be a Coxeter system, where $G$ is  an infinite indecomposable  Coxeter group
and $S=\{ s_0,s_1,\dots,s_k \}$.
Let $H=\langle s_{i_1},\dots,s_{i_l} \rangle$ be a proper standard subgroup of $G$.
Then $[G:H]=\infty$.

\end{lemma}

For the proof of Theorem~\ref{facets} we need to consider decomposable groups also. 
More precisely, we use the following corollary of Lemma~\ref{subgr}.

\begin{corollary}
\label{dec}
Let $\left(G,S\right)$ be a Coxeter system, where $G$ is an infinite Coxeter group
and $S=\{ s_0,s_1,\dots,s_k \}$. Suppose that for some $l\le k$ the group 
$G_l=\langle s_0,\dots,s_l \rangle$ is infinite and indecomposable, and let $m$ be a 
positive integer not exceeding $k$. Then a subgroup $H=\langle s_{m},\dots,s_{k} \rangle$ 
has infinite index in $G$.   

\end{corollary}

\begin{proof}
If $l=k$ then $G$ is indecomposable, and lemma~\ref{subgr} applies. So, we may assume that $G$ is 
decomposable. Permuting elements of $S$, we may assume that there exists $p$ such that a group $G_p$ is indecomposable and
commutes with all $s_i$ for $i>p$. Then the index $[G:H]$ is equal to the index $[G_p:H\cap G_p]$ 
which is infinite due to Lemma~\ref{subgr}.

\end{proof}

\setcounter{section}{1}
\setcounter{theorem}{0}

\begin{theorem}
\label{facets}
Let $\left(G,S\right)$ be a Coxeter system, where $G$ is infinite and indecomposable, and $S$ is finite.
If $P$ is a compact polytope in $\Sigma$,
then $|P|\ge |S|$.

\end{theorem}

\begin{proof}
Suppose that the theorem is broken. Then there exist 
a Coxeter system $\left(G,S\right)$ and a polytope $P$ in $\Sigma$ 
such that  $|P|< |S|$.
The proof is by induction on $|P|$, i.e. for any $k<|S|$ we assume that there is no 
compact polytope $P'\subset \Sigma$ such that $|P'|<k$ and prove that there is no
compact polytope $P\subset \Sigma$ such that $|P|=k$. 
(Since $G$ is infinite, this trivially holds for the cases $|P'|=0$ and $1$).

Suppose there exists a compact polytope $P\subset \Sigma$ such that $|P|=k<|S|$.
Recall that $\cal P$ is the set of all compact convex polytopes in $\Sigma$.
Define
$$ 
{\cal P}_1=\left\{P_1\in {\cal{P}} \ \left| \ P_1\subset P, \ |P_1|=k \right. \right\}.
$$

Since $P$ is a compact polytope, $P$ is intersected by finitely many walls of $\Sigma$, so
${\cal P}_1$ is a finite set.  ${\cal P}_1$ is not empty as it contains  $P$.
Let $P_{\mathrm{min}}\in {\cal P}_1$ be a polytope minimal with respect to inclusion.

\medskip
\noindent
{\bf Claim 1:} {\it $P_{\mathrm{min}}$ has no decomposed dihedral angle}.\\
Indeed, let $a$ and $b$ be facets of $P_{\mathrm{min}}$, and let $\mu$ be a wall decomposing a dihedral 
angle formed by $a$ and $b$. Then $\mu^+\cap P_{\mathrm{min}}\in \cal P$, $\mu^+\cap P_{\mathrm{min}}\subset P$, 
and  $|\mu^+\cap P_{\mathrm{min}}|\le k$ (since $\mu^+$ contains only one of the facets $a$ and $b$).
By the induction assumption, the case $|\mu^+\cap P_{\mathrm{min}}|< k$ is impossible.
Hence, $|\mu^+\cap P_{\mathrm{min}}|=k$, which contradicts the assumption that $P_{\mathrm{min}}$
is a minimal by inclusion polytope in ${\cal P}_1 $.

\medskip
Define the set of polytopes ${\cal P}_1'$ in the following way:
\begin{multline*}
{\cal P}_1'=\{P_1'\in {\cal{P}} \ | \ P_1'\subset P_{\mathrm{min}}, \ |P_1'|=k+1,\\
\text{ and all but one facet of } P_1' \text{ are facets of } P_{\mathrm{min}}  \}.
\end{multline*}

Clearly, ${\cal P}_1'$ is a finite set.
To show that ${\cal P}_1'$ is not empty notice, that $P_{\mathrm{min}}$ is not a chamber of $\Sigma$ 
since  $|P_{\mathrm{min}}|=k<|S|$, while a chamber has $|S|$ facets.
Therefore, there exists a wall $\mu$ decomposing $P_{\mathrm{min}}$ into two polytopes, namely
$P_{\mathrm{min}}\cap \mu^+$ and $P_{\mathrm{min}}\cap \mu^-$.
It is clear that $|P_{\mathrm{min}}\cap \mu^+|\le k+1$. The case $|P_{\mathrm{min}}\cap \mu^+|< k$ is impossible by the induction assumption,
the case $|P_{\mathrm{min}}\cap \mu^+|= k$ is impossible since $P_{\mathrm{min}}$ is a minimal 
by inclusion element of ${\cal P}_1$.
Hence, $|P_{\mathrm{min}}\cap \mu^+|=k+1$, so $P_{\mathrm{min}}\cap \mu^+ \in {\cal P}_1'$.  

Let ${{\cal P}'_1}_{\mathrm{min}}\subset{\cal P}_1'$ be the set of polytopes of ${\cal P}_1'$ which are minimal by inclusion, and let $P_{\mathrm{min}}'\in {{\cal P}_1'}_{\mathrm{min}}$.

\medskip
\noindent
{\bf Claim 2:} {\it $P_{\mathrm{min}}'$ has no decomposed dihedral angle}.\\
Again, let $a$ and $b$ be facets of $P_{\mathrm{min}}'$, and let $\mu$ be a wall decomposing 
the dihedral angle formed by $a$ and $b$. Then $\mu$ decomposes $P_{\mathrm{min}}'$ into two polytopes, namely
$P_{\mathrm{min}}'\cap \mu^+$ and $P_{\mathrm{min}}'\cap \mu^-$. The polytope 
$\mu^+\cap P_{\mathrm{min}}'\in \cal P$, $\mu^+\cap P_{\mathrm{min}}'\subset P_{\mathrm{min}}$, 
and  $|\mu^+\cap P_{\mathrm{min}}'|\le k+1$
(since $\mu^+$ does not contain either $a$ or $b$).
By the induction assumption, the case $|\mu^+\cap P_{\mathrm{min}}'|< k$ is impossible.
The case $|\mu^+\cap P_{\mathrm{min}}'|=k$ is impossible either, since  $P_{\mathrm{min}}$
is a minimal by inclusion polytope in ${\cal P}_{\mathrm{min}}'$.
Hence, $|\mu^+\cap P_{\mathrm{min}}'|=k+1$, which contradicts the assumption that $P_{\mathrm{min}}'$
is a minimal by inclusion polytope in ${\cal P}_{\mathrm{min}}'$.

\medskip

In particular, Claims 1 and 2 imply that all dihedral angles of $P_{\mathrm{min}}$ and $P_{\mathrm{min}}'$  
are equal to some dihedral angles of chamber of $\Sigma$, so $P_{\mathrm{min}}$ and $P_{\mathrm{min}}'$ are Coxeter polytopes.
Therefore, $\left(\Gamma_{P_{\mathrm{min}}'},S_{P_{\mathrm{min}}'}\right)$ is a Coxeter system for the group 
$\Gamma_{P_{\mathrm{min}}'}$, where $S_{P_{\mathrm{min}}'}$ is the complete collection 
of the reflections with respect to the facets of $P_{\mathrm{min}}'$,
and $\Gamma_{P_{\mathrm{min}}'}$ is  the group generated by all these reflections.
Similarly, $\left(\Gamma_{P_{\mathrm{min}}},S_{P_{\mathrm{min}}}\right)$ is a Coxeter system 
for the group $\Gamma_{P_{\mathrm{min}}}$, and $\Gamma_{P_{\mathrm{min}}}$ is a standard subgroup of 
$\Gamma_{P_{\mathrm{min}}'}$.

The group $\Gamma_{P_{\mathrm{min}}}$ may be decomposable. However, any maximal indecomposable component of 
$\Gamma_{P_{\mathrm{min}}}$ is infinite. Indeed, suppose that $G_0=\langle s_1,\dots,s_l\rangle$ is a maximal 
indecomposable standard subgroup of $\Gamma_{P_{\mathrm{min}}}$, and $G_0$ is finite. Consider the union of all
polytopes $gP_{\mathrm{min}}$, $g\in G_0$. This union is a compact Coxeter polytope with $k-l$ facets consisting of $|G_0|$
copies of $P_{\mathrm{min}}$. This contradicts the induction assumption that we have no compact polytopes in $\Sigma$ with less
than $k$ facets.

Now consider the facet $\mu$ of $P_{\mathrm{min}}'$ which is not a facet of $P_{\mathrm{min}}$, and let us prove the following statement.

\medskip
\noindent
{\bf Claim 3:} {\it there exists $\widetilde P_{\mathrm{min}}'\in{{\cal P}_1'}_{\mathrm{min}}$ such that $\mu$ is not orthogonal to all facets of $P_{\mathrm{min}}$.}\\ 
Suppose the contrary. Take any $P_{\mathrm{min}}'\in{{\cal P}_1'}_{\mathrm{min}}$, and denote by $s$ the reflection in $\mu$. Then $P_{\mathrm{min}}'\cup s(P_{\mathrm{min}}')=P_{\mathrm{min}}$, i.e. $\mu$ divides $P_{\mathrm{min}}$ into two congruent parts. Since $s$ commutes with all the other standard generators of $\Gamma_{P_{\mathrm{min}}'}$, the group $\Gamma_{P_{\mathrm{min}}'}$ is decomposable, which implies that it is a proper subgroup of $G$. In particular, $P_{\mathrm{min}}'$ is not a chamber of $\Sigma$, so there is a wall $\mu_1$ dividing $P_{\mathrm{min}}'$. 

Suppose that $\mu_1$ is not orthogonal to all the facets of $P_{\mathrm{min}}$. Then consider subset of ${\cal P}_1'$ consisting of polytopes lying inside $P_{\mathrm{min}}'\cap\mu_1^+$. This set is not empty since it contains $P_{\mathrm{min}}'\cap\mu_1^+$. Take a minimal by inclusion polytope ${P_{\mathrm{min}}^{'+}}$. It is cut of $P_{\mathrm{min}}$ by a wall $\mu'_1$. By assumption, $\mu'_1$ is orthogonal to all the facets of $P_{\mathrm{min}}$. Therefore, $\mu'_1\ne\mu_1$, and $P_{\mathrm{min}}$ consists of two copies of $P_{\mathrm{min}}^{'+}$. Now consider a subset 
of ${\cal P}_1'$ consisting of polytopes lying inside $P_{\mathrm{min}}'\cap\mu_1^-$, and take a minimal by inclusion polytope $P_{\mathrm{min}}^{'-}$. Clearly, $P_{\mathrm{min}}$ consists of two copies of $P_{\mathrm{min}}^{'-}$ either. However, $P_{\mathrm{min}}^{'-}$ does not intersect $P_{\mathrm{min}}^{'+}$, so a copy of $P_{\mathrm{min}}^{'-}$ should contain $P_{\mathrm{min}}^{'+}$, and a copy of $P_{\mathrm{min}}^{'+}$ should contain $P_{\mathrm{min}}^{'-}$. The contradiction shows that any wall dividing $P_{\mathrm{min}}$ is orthogonal to all facets of $P_{\mathrm{min}}$.

Now consider a chamber $F$ of $\Sigma$ contained in $P_{\mathrm{min}}$. We may assume that at least one facet of $F$ belongs to some face of  $P_{\mathrm{min}}$. Facets of $F$ are of two types: some of them belong to facets of $P_{\mathrm{min}}$, the others belong to walls dividing $P_{\mathrm{min}}$. Both sets are non-empty. However, any facet from one of these sets is orthogonal to any facet from another one. This implies that the group $G=\Gamma_F$ is not indecomposable, which contradicts the assumption of the theorem.      

\medskip  

By Claim~3, we may assume that the reflection $s$ in the facet $\mu$ of $\widetilde P_{\mathrm{min}}'$ does not commute with all the elements of $S_{P_{\mathrm{min}}}$. Consider the maximal indecomposable component $\Gamma$ of $\Gamma_{\widetilde P_{\mathrm{min}}'}$ containing $s$. Since all maximal indecomposable components of $\Gamma_{P_{\mathrm{min}}}$ are infinite, $\Gamma$ is also infinite. So, we are in assumptions of Corollary~\ref{dec}, and $[\Gamma_{\widetilde P_{\mathrm{min}}'}:\Gamma_{P_{\mathrm{min}}}]=\infty$. This implies that $P_{\mathrm{min}}$ contains infinitely many copies of $\widetilde P_{\mathrm{min}}'$, thus,  $P_{\mathrm{min}}\subset P$  contains infinitely many chambers of $\Sigma$.
This contradicts the assumption that $P$ is a compact polytope in $\Sigma$.

\end{proof}

\begin{theorem}
\label{gen}
Let $\left(G,S\right)$ be a Coxeter system, where
$G$ is infinite and indecomposable, and $S$ is finite.
Let $H\subset G$ be a finite index reflection subgroup.
Then any set of reflections generating $H$ contains at least $|S|$ elements.

\end{theorem}

\begin{proof}
Let $\left(H,S'\right)$ be a Coxeter system, where $S'$ consists of some 
reflections of $\left(G,S\right)$. By Lemma~\ref{generators}, 
any set of generating reflections of $H$ contains at least $|S'|$ reflections.
So, we are left to show that $|S'|\ge |S|$.

Consider a fundamental chamber $P$ of the subgroup $H$ acting on the complex $\Sigma=\Sigma\left(G,S\right)$. 
Since $H$ is a finite index subgroup, $P$ is a compact polytope in $\Sigma$. 
By Theorem~\ref{facets}, this implies that $|P|\ge |S|$.
Since, $|P|=|S'|$, we obtains the required inequality.

\end{proof}

\setcounter{section}{3}
\setcounter{theorem}{2}

\begin{remark}
\label{nes}
The conditions for $G$ to be infinite and indecomposable are essential in both Theorems~\ref{facets} and~\ref{gen}.
Clearly, any finite group contains a trivial subgroup of finite index. Concerning decomposable groups, the group
$$\left\langle s_1,s_2,s_3\ |\ s_1^2=s_2^2=s_3^2=(s_1s_3)^2=(s_2s_3)^2=1\right\rangle$$ contains a subgroup  
$\left\langle s_1,s_2\ |\ s_1^2=s_2^2=1\right\rangle$ of index two.

\end{remark}

\section{Andreev's theorem}
\label{andr}

In this section we prove a counterpart of Andreev's theorem for facets of an acute-angled polytope 
in the complex $\Sigma$.

\begin{remark}
\label{andr for cox}
Notice that in partial case when $P$ is a chamber, the statement is tautological even 
for arbitrary faces: by construction of $\Sigma$ a collection of walls 
(respectively, facets of chamber) has a non-empty intersection if and only if the 
corresponding reflections generate a finite group. It is easy to see that the same 
is true if $P$ is any Coxeter polytope in $\Sigma$: for that it is sufficient to consider
$P$ as a chamber of corresponding subgroup $\Gamma_P$ of $G$.

\end{remark} 

\begin{lemma}
\label{stacan}
Let $P_1$ and $P_2$ be convex polytopes in $\Sigma$ and let $c$ be a common facet of 
$P_1$ and $P_2$. If all the dihedral angles of $P_1$ and $P_2$ formed by $c$ with other facets of  
$P_1$ and $P_2$ are acute, then $P=P_1\cup P_2$ is a convex polytope.

\end{lemma}

\begin{proof}
Let $P_1=\gamma^+\cap\left(\bigcap\limits_{i=1}^k \alpha_i^+\right)$
and $P_2=\gamma^-\cap\biggl(\bigcap\limits_{j=1}^l \beta_j^+\biggr)$,
where $\gamma$ is a wall containing $c$, $\alpha_i$ is a wall containing a facet $a_i$ of $P_1$,
and $\beta_j$ is a wall containing a facet $b_j$ of $P_2$.
We will prove that $P=P_1\cup P_2= 
\left(\bigcap\limits_{i=1}^k \alpha_i^+\right)\cap \biggl(\bigcap\limits_{j=1}^l \beta_j^+\biggr)$. 
To prove this, it is sufficient to show that $P_2\subset \alpha_i^+$ for all $i=1,\dots,k$,
and $P_1\subset \beta_j^+$ for all $j=1,\dots,l$. We prove the former of these statements,
the latter may be shown in the same way.
Given a facet $a_i=\alpha_i\cap P_1$, we consider two cases: either $\alpha_i\cap \gamma$ contains a face of $P$
or it does not.

At first, consider a facet $a=a_i$ such that $\alpha\cap\gamma $ contains a face of $P$
(where $\alpha$ is the wall containing $a$).
Then there exists a facet $b=b_j$ of $P_2$ such that  
$\left(a\cap c\right)\cap P=\left(b\cap c\right)\cap P$. Denote by $\beta$ the wall containing $b$.
To prove that  $P_2\subset \alpha^+$, it is enough  to prove that $\gamma^-\cap \beta^+\subset \alpha^+$.
For this, it is sufficient to show that $\vn\left(\gamma^-\cap \beta^+\right)\cap \alpha =\emptyset$
(since $c\subset \alpha^+$).
Since $\gamma$ divides $\Sigma$ into two connected components,
$\gamma$ divides $\alpha$ into $\alpha\cap \gamma^+$ and  $\alpha\cap \gamma^-$.
The component $\alpha\cap \gamma^+$ does not intersect   $\vn\left(\gamma^-\cap \beta^+\right)$
since these two sets belong to different halfspaces with respect to $\gamma$.
The component $\alpha\cap \gamma^-$ does not intersect   $\vn\left(\gamma^-\cap \beta^+\right)$ since 
$\alpha\cap \gamma^-$ belongs to $\beta^-$ (the latter statement follows immediately
while considering any maximal cell of $\Sigma$
intersecting $c\cap b\cap a$).

\medskip

Thus, we have proved that  $P\subset \alpha_i^+$ for all $i\in I$, where 
$i\in I$ if and only if $\left(\gamma\cap \alpha_i\right)\cap P\ne \emptyset$.
Consider $Q=\bigcap\limits_{i\in I} \alpha_i^+$.
Notice that $P_2\subset Q\cap \gamma^-$, and $\gamma\cap Q= \gamma\cap P_1=c$.

Now, consider a facet $a_j$ such that $j\notin I$.
Suppose that $\alpha_j\cap \vn\left(P_2\right)\ne \emptyset$. Then $\alpha_j \cap \left(Q\cap \gamma^-\right)$ is not empty either.
Consider two points $p_1\in \rel \left(P_1\cap \alpha_j\right)$ and $p_2\in \vn\left( P_2\right)\cap \alpha_j$.
Here we require $p_1$ to belong to the relative interior of $P_1\cap \alpha_j$ to be sure that $p_1$
lies in the interior of $Q$.
Clearly, $p_2\in \vn\left(Q\right)$, too.
Notice that there is one special case when we cannot take  $p_1\in \rel \left(P_1\cap \alpha_j\right)$: this happens if $a_j$
is a point. However, in this case $\alpha_j$ is also a point, so it cannot intersect $P_2$.    

Consider a geodesic segment $\xi$ joining $p_1$ with $p_2$. 
Since $\alpha_j$ is a wall, and any wall is totally geodesic, $\xi\subset \alpha_j$.
On the other hand,
since $Q$ is a convex polytope and
$p_1,p_2\in \vn\left(Q\right)$, we see that $\xi\subset \vn\left(Q\right)$. 
Moreover, $p_1\in \gamma^+$, and $p_2\in \gamma^-$, so, there exists a point $p\in \xi$
such that $p\in \gamma$. Since $\xi\subset \vn\left(Q\right)$, we obtain that $p\in \vn\left(Q\right)\cap \gamma=\vn\left(c\right)$.
In particular, $\alpha_j\cap \vn\left(c\right)\ne \emptyset$, which is impossible by definition
of convex polytope $P_1$.  
Therefore, $\alpha_j\cap \vn\left(P_2\right)=\emptyset$, and the lemma is proved. 

\end{proof}

\begin{lemma}
\label{andr for facets}
Let $P$ be an acute-angled polytope in $\Sigma$.
Let $a$ and $b$ be facets of $P$ and $\alpha$ and $\beta$ be the walls containing 
 $a$ and $b$ respectively.
If $a\cap b=\emptyset$ then $\alpha\cap \beta=\emptyset$.

\end{lemma}

\begin{proof}
Let $r_{\alpha}$ and $r_{\beta}$ be the reflections with respect to $\alpha$ and ${\beta}$. 
Suppose that $\alpha\cap {\beta}\ne \emptyset$. Then, by construction of $\Sigma$, 
$r_{\alpha}$ and $r_{\beta}$ generate  a finite dihedral group.

Consider a sequence of polytopes $P_0=P$, $P_1=r_{\alpha}P_0$, $P_2=r_{\alpha}r_{\beta}r_{\alpha}P_1
=r_{\alpha}r_{\beta}P_0$,
$P_3=\left(r_{\alpha}r_{\beta}\right)r_{\alpha}\left(r_{\alpha}r_{\beta}\right)^{-1}P_2 =r_{\alpha}r_{\beta}r_{\alpha}P_0$, and so on, i.e. 
$P_i=r_{\alpha}r_{\beta}P_{i-2}$. By construction, $P_i$ and $P_{i+1}$ have a common facet,
and $P_i$ is symmetric to $P_{i+1}$ with respect to this facet. 
Notice that each of $P_i$ is an acute-angled polytope (as an image of an acute-angled polytope).
Let $Q_0=P_0$ and $Q_i=Q_{i-1}\cup P_i$.

We claim that 
$Q_i$ is a convex polytope and any dihedral angle of $Q_i$ formed by $c_i=P_i\cap P_{i+1}$
and any other facet is acute. 
The proof is by induction on $i$.
Indeed, $Q_0=P$ is a convex acute-angled polytope.
Suppose that the statement is true for $Q_{i-1}$. Then Lemma~\ref{stacan} implies that $Q_i$ is a convex polytope.
Any dihedral angle of $Q_i$ formed by $c_i$
and any other facet is a dihedral angle of  an acute-angled polytope $P_i$, and hence, is acute.
Therefore, $Q_i$ is a convex polytope for any $i$. 

Notice, that the wall $\gamma_{i-1}$ (containing the facet $c_{i-1}$ of $Q_{i-1}$)
decomposes $Q_i$ into two convex polytopes $Q_i\cap \gamma_{i-1}^+=Q_{i-1}$ and $Q_i\cap \gamma_{i-1}^-=P_i$,
and hence, $\vn\left(P_i\right)$ does not intersect $\vn\left(Q_{i-1}\right)$.
On the other hand, $r_{\alpha}r_{\beta}$ has finite order, and hence,  there exists $j\in \N$ such that $P_j=P_0$.
The contradiction competes the proof of the lemma.

\end{proof}

\section{The case of equal ranks}
\label{nn}

Throughout this section we suppose that $G$ is an infinite indecomposable Coxeter group with a finite set $S$
of standard generators, $H\subset G$ is a finite index reflection subgroup with a set $S'$ of standard generators, 
and the ranks of $G$ and $H$ are equal (i.e. $|S'|=|S|$). Let $|S'|=k$.

\begin{lemma}
\label{nerve}
The nerve of a Coxeter system $\left(H,S'\right)$ can be obtained from the nerve of $\left(G,S\right)$  
by deleting some simplices.
\end{lemma}

\begin{proof}
Let $P$ be a fundamental chamber of $H$. Define the set of polytopes ${\cal P}(P)$ in the following way:
$$
{\cal P}(P)=\{P'\in {\cal{P}} \ | \ P'\subset P,
\text{ all but one facet of } P' \text{ are facets of } P\}.
$$

${\cal P}(P)$ is finite and non-empty: since $P$ is not a chamber of $\Sigma$,  at least one wall of $\Sigma$
decomposes  $P$ into two smaller polytopes, each of them belongs to ${\cal P}(P)$. Thus, the set ${\cal P}_{\mathrm{min}}(P)$ 
of polytopes minimal in ${\cal P}(P)$ by inclusion is not empty. 


\medskip
{\bf Claim:} {\it there exists ${ P}_1\in{\cal P}_{\mathrm{min}}(P)$, such that ${ P}_1$ is a Coxeter polytope, and $|{ P}_1|=k$.}

The first statement is evident: any polytope ${ P}_1'\in{\cal P}_{\mathrm{min}}(P)$ is a Coxeter polytope. 
Indeed, let $\mu_1$ be the facet of $P_1'$ which is not a facet of $P$. If ${ P}_1'$ is not a Coxeter polytope, 
a dihedral angle formed by $\mu_1$ and some facet of $P$ is decomposed by some wall $\mu_1'$. 
Then $\mu_1'$ cuts out of $P$ a polytope which is contained in  ${ P}_1'$, which contradicts minimality of  ${ P}_1'$.

By Theorem~\ref{facets}, any ${ P}_1'\in{\cal P}_{\mathrm{min}}(P)$ has at least $k$ facets. If it has more than $k$ facets 
(i.e., $k+1$ facets), then all the facets of $P$ are facets of ${ P}_1'$, so $H$ is a standard subgroup of the group 
$\Gamma_{P_1'}$ which is generated by reflections in the facets of ${ P}_1'$. 

The further proof of the claim is similar to the proof of Claim~3 from the proof of Theorem~\ref{facets}.

Suppose that any ${ P}_1'\in{\cal P}_{\mathrm{min}}(P)$ has exactly $k+1$ facets. It follows from Theorem~\ref{facets} that any indecomposable 
component of $H$ is infinite. So, if for some ${ P}_1'\in{\cal P}_{\mathrm{min}}(P)$ the corresponding facet $\mu_1$ is not 
orthogonal to all the facets of $P$, we are in assumptions of Cor.~\ref{dec} for the groups $\Gamma_P\subset\Gamma_{{ P}_1'}$. 
Therefore, we may assume that for any ${ P}_1'\in{\cal P}_{\mathrm{min}}(P)$ the corresponding facet $\mu_1$ is orthogonal to all the 
facets of $P$, and $P$ consists of two copies of ${ P}_1'$. In particular, $\Gamma_{{ P}_1'}$ is decomposable, so ${ P}_1'$ is 
not a chamber of $\Sigma$.

Now take any wall $\mu_2$ dividing ${ P}_1'$, and suppose that $\mu_2$ is not orthogonal to all the facets of $P$. Then consider two polytopes $P^+,P^-\in{\cal P}_{\mathrm{min}}(P)$ lying inside $P\cap\mu_2^+$ and $P\cap\mu_2^-$ respectively. $P$ consists of two copies of each of these polytopes, however, they do not intersect, so a copy of $P^-$ should contain $P^+$, and a copy of $P^+$ should contain $P^-$. The contradiction implies that any wall dividing $P$ is orthogonal to all the facets of $P$. Then, as in the Claim~3 from the proof of Theorem~\ref{facets}, take a chamber $F$ of $\Sigma$ contained in $P$ and show that the group $G=\Gamma_F$ is not indecomposable. The contradiction completes the proof of the claim.   

\medskip

The claim above implies that we may take ${ P}_1\in{\cal P}_{\mathrm{min}}(P)$ such that the corresponding facet $\mu_1$ 
does not intersect exactly one facet of $P$, say $f_1$. All the remaining facets of $P$ are facets of ${ P}_1$. 
Therefore, we have a one-to-one correspondence between 
facets of $P$ and ${ P}_1$: facet $\mu_1$ corresponds to $f_1$, and any other facet corresponds to itself. This implies 
a correspondence of vertices of the nerves of $H$ and $\Gamma_{P_1}$. We want to prove that the nerve of $H$  
can be obtained from the nerve of $\Gamma_{P_1}$  by deleting some simplices. For this we show that if a collection 
of facets of $P$ has a non-empty intersection in $P$, then the corresponding collection of facets of ${ P}_1$ (i.e. 
substituting   $\mu_1$ by $f_1$) has a non-empty intersection in ${ P}_1$. 

Suppose that a collection $J=\{f_{i_1},f_{i_2},\dots,f_{i_n}\}$ of facets of $P$ defines a face of $P$, i.e.
the intersection of all the facets contained in $J$ is not empty. If $f_1\notin J$, then 
the facets contained in $J$ have non-empty intersection in  ${ P}_1$ (see Remark~\ref{andr for cox}).     
If $f_1\in J$, then consider the face $f$ of $P$ defined by the collection $J\setminus f_1$. By our assumption, 
$f$ intersects $f_1$. On the other hand, it was shown above that $f$ contains a face of ${ P}_1$. Since 
${ P}_1$ and $f_1$ are contained in distinct halfspaces with respect to $\mu_1$, this implies that
$f$ intersects $\mu_1$.       

\medskip

Therefore, the nerve of $H$ can be obtained from the nerve of $\Gamma_{P_1}$  by deleting some simplices.
Now, substituting $P$ by ${ P}_1$ in the construction above, we may choose a polytope 
${ P}_2$ with $k$ facets which is minimal in ${\cal P}(P_1)$. Since ${ P}_2\ne { P}_1$, ${ P}_2$ contains smaller number of
chambers of $\Sigma$ than ${ P}_1$ does. Following this procedure (i.e. ${ P}_3$ is minimal in ${\cal P}(P_2)$ 
and so on), we see that for some $m$ the polytope ${ P}_m$ is a chamber of $\Sigma$, so $\Gamma_{P_{m}}=G$. 
The same proof as above shows that for any $i$ the nerve of $\Gamma_{P_i}$ can be obtained from the nerve of 
$\Gamma_{P_{i+1}}$ by deleting some simplices. Hence, the nerve of $H$ also can be obtained from the nerve of 
$G$ by deleting some simplices. 

\end{proof}

\begin{example}
As an example of the situation described in Lemma~\ref{nerve}, consider a group $\Gamma=\Gamma(2,3,\infty)$ generated by reflections in
the sides of hyperbolic triangle with angles $\pi/2,\pi/3$ and $0$. The nerve $N$ of the group consists of $3$ vertices
joined by two edges, corresponding to dihedral groups of order $4$ and $6$. This group has exactly three reflection subgroups 
of rank $3$. The subgroup $\Gamma_1=\Gamma(3,3,\infty)$ of index two is generated by reflections in the sides of triangle with angles $\pi/3,\pi/3$ and $0$. Its nerve $N_1$ is isomorphic to $N$. The subgroup $\Gamma_2=\Gamma(2,\infty,\infty)$ of index three is generated by reflections in the sides of triangle with angles $\pi/3,0$ and $0$. Its nerve $N_2$ consists of $3$ vertices, only two of which joined by an edge corresponding to dihedral group of order $4$. The subgroup $\Gamma_3=\Gamma(\infty,\infty,\infty)$ of index six is generated by reflections in the sides of ideal triangle. Its nerve $N_3$ consists of $3$ vertices and no edges.

\end{example}

Finally, we provide a necessary condition for $\left(G,S\right)$  to have a finite index reflection subgroup of the same rank.

\begin{lemma}
\label{comm}
Let $(G,S)$ be a Coxeter system, where $G$ is infinite indecomposable, and $S$ is finite. 
Suppose that there exists a finite index reflection subgroup $H$ of $G$ of rank $|S|$. 
Then there exists $s_0\in S$ such that at least one of the following holds:

(1) $s_0$ commutes with all but one elements of $S$;

(2) the order of $s_0s$ is finite for all $s\in S$.
\end{lemma}

\begin{proof}
Let ${\cal P}_k$ be the set of polytopes in $\Sigma$ with exactly $k$ facets. 
Let $P_0\in{\cal P}_k$ be any polytope containing no elements of  ${\cal P}_k$ except
chambers of $\Sigma$, and $P_0$ itself is not a chamber of $\Sigma$. 
Since the rank of $H$ is equal to the rank of $G$, such $P_0$ does exist (for example, one can take as $P_0$
a fundamental domain of $H$ in $\Sigma$). 
Clearly, there is at least one wall of $\Sigma$ which divides $P_0$.  

Suppose that $P_0$ contains a decomposed dihedral angle formed by facets $f_1$ and $f_2$ of $P_0$.
Let $\mu$ be a wall of $\Sigma$ which decomposes that angle. Then $\mu$ decomposes $P_0$ into 
two polytopes $P_0^+$ and $P_0^-$, each of them has at most $k$ facets 
(since one of $f_1$ and $f_2$, say $f_2$, is not a facet of $P_0^+$, 
and the other is not a facet of $P_0^-$). By Theorem~\ref{facets}, each of  $P_0^+$ and $P_0^-$ has exactly $k$ facets.
Thus,  $P_0^+$ and $P_0^-$ are chambers of $\Sigma$, and $P_0^+$ can be obtained from $P_0^-$ by reflecting in 
$\mu$. Since $k-1$ of $k$ facets of  $P_0^+$ are facets of $P_0^-$, a facet $\mu\cap P_0$ of $P_0^+$ 
is orthogonal to all but one facets of $P_0^+$, so the condition (1) holds. A unique facet of $P_0^+$ which 
is not orthogonal to $\mu\cap P_0$ is $f_1$. By assumption, $\mu$ intersects $f_1$, so in this case the 
condition (2) also holds.    

Now suppose that $P_0$ contains no decomposed dihedral angles. In particular, $P_0$ is a Coxeter polytope.
As in the proof of Lemma~\ref{nerve}, consider the set ${\cal P}(P_0)$, and take a minimal (by inclusion) 
element $P_1$. Again, $P_1$ is a Coxeter polytope, and $|{ P}_1|=k$. By the choice of $P_0$, this implies that 
$P_1$ is a chamber of $\Sigma$. Let $\mu$ be the wall of $\Sigma$ which contains a facet of $P_1$ but contains no
facets of $P_0$. Let $P_2=P_0\setminus P_1$. By Theorem~\ref{facets}, $P_2$ has at least $k$ facets. It is also clear 
that  $P_2$ has at most $k+1$ facets.

If $P_2$ has $k+1$ facets, then any facets of $P_0$ contains a facet of $P_2$. Therefore, any wall containing
a facet of $P_1$ contains a facet of $P_2$. Since ${ P}_1$ and $P_2$ are contained in distinct halfspaces 
with respect to $\mu$, this implies that $\mu$ intersects all facets of $P_1$, so condition (2) holds.

If $P_2$ has $k$ facets, then it is also a chamber of $\Sigma$, and $P_2$ can be obtained from $P_1$ 
by reflecting in $\mu$. Thus, the number of facets of $P_0$ is the number of facets of $P_1$ that are 
orthogonal to $\mu$ plus twice the number of remaining facets of $P_1$ except $\mu\cap P_0$. Solving this
linear equation, we see that $k-2$ facets of $P_1$ are orthogonal to $\mu$, so  condition (1) holds.         

\end{proof}

\begin{remark}
Let $(G,S)$ be a Coxeter system satisfying condition $(1)$ of Lemma~\ref{comm}. 
If the Coxeter relation between $s_0$ and $s'$ (which is a unique generator not commuting with 
$s_0$) has an even exponent (or if there is no relation), then the first condition of Lemma~\ref{comm} 
is also sufficient: the set $\{S\setminus s_0,s_0s's_0\}$ is a set of standard generators for a subgroup of index $2$. 
However, in case of odd exponent the first condition is not sufficient: a group generated by reflections in sides of 
hyperbolic triangle with angles $(\pi/2,\pi/5,\pi/5)$ has no finite index reflection subgroups of rank $3$ 
(see, for example,~\cite{KS} or~\cite{F}). The same example shows that the second condition is not sufficient, either: 
any two standard generators of the group above generate a finite dihedral group. 

\end{remark}

\end{document}